\documentclass[12pt]{article}
\usepackage[english]{babel}
\usepackage[cp1251]{inputenc}

\usepackage{amsthm}
\usepackage{hhline}
\usepackage{amscd}

\usepackage{amsfonts, amscd, amsmath, amssymb}
\usepackage[dvips]{graphicx}
\usepackage[cp1251]{inputenc}
\usepackage[T2A]{fontenc}	

\usepackage{amsthm}
\usepackage{amssymb}
\usepackage{amsmath,amsthm,amsfonts}
\usepackage{hhline}

\newtheorem{theorem}{Theorem}[section]
\newtheorem{lemma}{Lemma}[section]
\newtheorem{rem}{Remark}[section]
\newtheorem{cor}{Corrolary}[section]

\begin{document}
\title{ Curvature forms and Curvature functions for 2-manifolds with boundary}
\author{Kaveh Eftekharinasab
\footnote{Institute of mathematics of Ukrainian Academy of Sciences.
E-mail : kaveh@imath.kiev.ua}}
\date{}
\maketitle
\begin{abstract}
We obtained that any 2-form and any smooth function on 2-manifolds with boundary can be realized as the
 curvature form and the Gaussian curvature function of some  Riemannian metric, respectively.

\end{abstract}

\section{Introduction}
\indent

For 2-manifolds, possibly, with boundary the classical Gauss Bonnet formula asserts a relationship
 between the Euler characteristic of a manifold and its Gaussian curvature and the geodesic curvature of
 the boundary. This is the only known obstruction on a given 2-form on a manifold to be the curvature form
 of some Riemannian metric. Nevertheless, it imposes a constraint on the sign of a function for being 
the curvature function of a metric. The problem of prescribing curvature forms on closed 2-manifolds 
was solved by Wallach and Warner \cite{warner}. They showed that the Gauss Bonnet formula is a necessary 
and sufficient condition on a 2-form to be a curvature form. Later, the problem of prescribing curvature 
functions has been studied by some authors and completely solved for closed manifold by Kazadan and
 Warner \cite{kazdan}. They proved that any smooth function which satisfies Gauss Bonnet sign condition 
is the Gaussian curvature of some Riemannian metric.

In this paper we deal with 2-manifolds with boundary and the problems of prescribing curvature forms and and
 curvature functions . In contrast with the case when manifolds have nonempty boundary no obstruction on 
2-forms and functions arises. It turns out that any 2-form and smooth function can be realized as the 
curvature form and curvature function of a metric respectively, this is a surprising phenomena.
\section{Preliminaries And The Main Results}

If we want to study manifolds with boundary we  often face with the problems of extensions and restrictions
of smooth objects, 
we handle these problems by gluing manifolds together, providing desired extensions using the
 elementary techniques of differential topology. At first, we shall consider the problem of realization of forms then the same 
method will be used for functions.\bigskip \\
In the case of 2-manifolds with boundary two problems arise: the first question is as the same as 
the question addressed by Wallach and Warner, i.e. which $2-$form on $M$ can be
the curvature form of some Riemannian metric? In this paper we answer to the question by employing the typical techniques 
of differential topology. Another question is more difficult and still is unsolved: Given a $2-$form $\Omega$ on $M$ and a $1-$form
$\Psi$ on $\partial M$, is there any metric with $\Omega$ and $\Psi$ as its curvature form and geodesic form,
respectively? It is believed that this problem can be solved by the Hodge decomposition theorem for manifolds with
boundary  resembles to the proof of Wallach and Warner in the case of manifolds without boundary.  

Suppose $M$ and $N$ are manifolds with smooth boundary. The existence of a collar neighborhood of the boundary
enables us to construct smooth extensions 
of maps which are defined only on the boundary to the whole manifold. Also we can glue smooth maps so that it agrees with a smooth structure 
on the constructed manifold. If $f_1$ and $f_2$ are smooth maps 
 defined on $M$ and $N$, respectively, then
we can piece them together to get a smooth map $H = f_1 \cup f_2$ on $ W = M \cup_{\Phi} N$, here $\Phi$ is a gluing
diffeomorphism, see~\cite [Lemma 3.7]{milnor} and~\cite [Theorem 2.8]{mo}. The notion $H = f_1 \cup f_2$ 
is a bit deceiving. The map $H$ may not quite restricts to
$f_1$ on $M$ nor $f_2$ on $N$, and so we need to modify $f_1$ and $f_2$ in respective
 collar neighborhood of boundaries.  In future when we glue smooth maps $f_1$ and $f_2$ and write $ f_1 \cup f_2$
we keep in mind that the obtained map is the map which is obtained by modifying $f_1$ and $f_2$ in collar
neighborhoods of their boundaries.\newpage
 Let $M$ be a connected, compact and oriented 2-manifold with smooth boundary. Now, glue 2-disk
 $D^{2}$ to $M$ to get a 2-manifold without boundary $\widetilde{M}$, suitably oriented, joined
 together along boundaries. Now we shall have occasion to extend forms from $M$ to the whole manifold: 
Suppose $\omega_{1}$ and $\omega_{2}$ are given 2-forms on $M$ and $D^{2}$ respectively,
(here we just consider 2-forms but in general it is true for arbitrary forms) 
which are locally represented as $\omega_{1}=f_{12}\:dx^1 \wedge\ dx^2$ and 
$\omega_{2}=g_{12}\:dy^1\wedge dy^2$ in collar neighborhoods of their boundaries, then smoothly piece together
  $f_{12}$ and $g_{12}$ in collar  
neighborhoods. In result we obtain a smooth function, say $h_{12}$, 
and consequently an $2$-form 
$\omega = h_{12}\,\,x_1\wedge x_2 $
on $\widetilde{M}$. 

We emphasize that the restrictions of $\omega$ to boundaries can not be
identified with $2-$forms on boundaries, rather they are smooth sections in restricted $2-$form bundles
$\bigwedge^2(M)|_{\partial M}$ and $\bigwedge^2(D^{2})|_{\partial D^{2}}$
 
\begin{lemma}
Let $\omega$ be a given $2-$form on $M$. Then for any arbitrary nonzero real number $\theta$ there exists an extension $\bar{\omega}$ 
of $\omega$ to $D^2$ such that 
$\int_ {D^{2}} \bar{\omega}= \theta.$
\end{lemma}
\begin{proof}
Let $\widetilde{\omega}$ be an arbitrary extension such that 
$\int_ {D^2}\widetilde{\omega}\neq 0$. We construct 2-form $ \bar{\omega} $ 
using bump function such that in an open neighborhood of the boundary coincides with
 $\widetilde{\omega}$ and $\int_ {D^{2}} \bar{\omega}= \theta$. 

Let $U$ be an open neighborhood 
of the boundary and $V$ be an open neighborhood of the boundary possibly smaller.
 Let $$f dx^{1} \wedge dx^{2}$$ be a local
 representation of $\widetilde{\omega}$ in $U$. Choose a smooth bump function $g$ supported 
in $U$ which is identically 1 in a neighborhood $V$ of the boundary.
Define
\begin{equation*}
\tilde{f}(x) = 
      \left\{ \begin{array}{l} 
          f(x)g(x),\qquad x\in U, \\ 0, \qquad \qquad \quad \, \mbox{otherwise}.
           \end{array} \right.
\end{equation*}
$\tilde{f}$ is smooth on $U$. If $x \notin U$, then $x$ does not belong to the support of $f$,
 hence there is an open set containing $x$ on which $\tilde{f}$ is $0$,
 because the support of $f$ is closed. Thus $\tilde{f}$ is smooth everywhere other than 
$U$ as well. Finally, since $\tilde{f}$ equals the identity on $V$, it follows $\tilde{f}$
 coincides with $f$ on $V$. 
Put $$\widehat{\omega} = \tilde{f} \,\widetilde{\omega},$$ and assume 
$$\int_{D^2} \widehat{\omega}= k\neq 0,\, 
\int_{U}\widehat{\omega}=k_1$$
and  $$ \int_{\Omega} \widehat{\omega}= k_2,\, 
 \int_{\widetilde{D^2}} \widehat{\omega}=k_3.$$
 Where $\Omega$ is the space 
between $U$ and $V$. Now define a new function
\begin{equation*}
h(x) = 
      \left\{ \begin{array}{l} 
          \mbox{identity} ,\qquad \qquad \,\, x\in V, \\ \dfrac{a-k_{1}}{k_{2}+k_{3}} g, \qquad \qquad  \mbox{elsewhere}.
           \end{array} \right.
\end{equation*}
Obviously, $h$ is smooth. Set $$\bar{\omega} = h \widehat{\omega}.$$ 
(Notice that we always can choose neighborhoods and function $g$ such that $k_2+k_3\neq 0$). $\bar{\omega}$ 
is the desired extension because it coincides with $\tilde{w}$ on an open neighborhood of the boundary this
means
it is smoothly extended and $\int_ {D^{2}} \bar{\omega}= \theta$.  
\end{proof}

As an evident consequence of this lemma we have the following corollary.
\begin{cor}\label{corollary 1}
For any 2-form $\omega$ on $M$ there exists an extension $\widetilde{\omega}$ such that 
$$\int_{\widetilde{M}} \widetilde{\omega}=2 \pi \chi \widetilde{M})$$ 
\end{cor}

\begin{theorem}\label{Theorem 1}
Let M be a connected, compact and oriented 2-manifold with smooth boundary. Then any 2-form $\omega$ on $M$ is the curvature form of some Riemannian metric $g$ on $M$.
\end{theorem}

\begin{proof}
By Corollary~\ref{corollary 1} there exits an extension $\widetilde{\omega}$ of ${\omega}$ such that
 $$ \int_{\widetilde{M}} \widetilde{\omega}  = 2 \pi \chi(\widetilde{M}), $$  
then by employing the theorem of Wallach and Warner~\cite{warner} for $\widetilde{\omega}$, we find a 
Riemannian metric $\widetilde{g}$ on $\widetilde{M}$ which its restriction to M is the expected metric.
\end{proof}
\begin{rem}
 On the boundary the curvature 2-form vanishes but the boundary value of given 2-form  $\omega$ computes by
pull-back $\jmath^* \omega$ ($\jmath : \partial M \hookrightarrow M $ is canonical inclusion). However, the metric $g$ on $M$ which its existence is guaranteed by the theorem with $\omega$ as its 2-curvature
form canonically induces a metric $\jmath^* g$ (ordinary distance function) to the boundary. The induced
metric determines the geodesic form $\Psi_{(\jmath^* g)}^1$ explicitly.
The geodesic form $\Psi_{(\jmath^* g)}^1$ is any 1-form $\varPhi$ on the boundary which satisfies 
$$\int_{\partial M}\varPhi = 2 \pi \chi(M) - \int_{M} \omega.$$
\end{rem}
\begin{rem}
Note that in what discussed and follows we just consider manifolds having only one boundary
 component, but, in general, when boundary consists of more than one component the theorems remain valid,
 we just need to glue $D^2$  to each component to get a closed manifold.
\end{rem}

Since in fact, we integrate a function not a 2-form, we may proceed with the same approach
 and expect the similar result for functions. however, again two problems arise: determining the Gaussian curvature and the geodesic curvature of
the boundary simultaneously, i.e. two smooth functions $f$ and $g$ are given, $f$ defined only on the boundary 
and $g$ defined on the whole manifold; we want to find a metric with $f$ as its geodesic curvature and $g$ as
its Gaussian curvature. This problem partially answered when $\chi(M) \leq 0$ in~\cite[Theorem 3]{cherrier}.
The problem which we concern  here is that given a smooth function $h$ on the whole manifold we look for a metric
which has $h$ as it Gaussian curvature and induced metric on the boundary determines the geodesic curvature of the boundary.

Suppose $M,\,\widetilde{M}$ are as before, and $f :M \rightarrow \mathbb{R} $ is smooth.
At a boundary point $p \in \partial M$, $f$ is smooth if there is a chart $(U,\phi)$ about $p$ such that
$f \circ \phi^{-1}$ is smooth at $\phi(p) \in \mathbb{H}^2$. The latter means that $f \circ \phi^{-1}$ has 
a smooth extension to a neighborhood of $\phi(p)$ in $\mathbb{R}^2$. 

\begin{rem}
 In~\cite{pl} it was addressed the the same problem of 2-manifold with boundary but in LP  category. In the
LP category the problem has different aspect, the problem has combinatorial character which causes to
use methods completely different from those used in smooth case.
\end{rem}
\begin{lemma}
Let $f$ be a smooth function defined on $M$. Then there exists an extension $\widetilde{f}$ to $\widetilde{M}$ such that
satisfies the Gauss-Bonnet sign conditions on $\widetilde{M}$.
\end{lemma}

\begin{proof} let $\bar{f}$ be an arbitrary extension which is not zero everywhere. Suppose $\chi(M) > 0$,
if there exists a point $x_0$ at which $f(x_0) > 0$ there is nothing to do. Otherwise, multiply $f$ to
a smooth function $g$, where
$$
\begin{array}{lll}
g=\,
$$
\left\{
\begin{array}{ll}
1, \qquad \quad \quad \mbox{in an open neighborhood of the boundary},\\
\mbox{negative}, \quad \mbox{at some point},
\end{array} \right.
$$
\end{array}
$$
Obviously, $fg$ is smooth and so is a desired extension. If $\chi(M) < 0$ we can modify the extension likewise.
If $\chi(M) = 0$ and $f$ does not vanish identically and does not change sign, it is strictly positive or 
negative, thereby we just need to multiply it to a smooth function which is equal to the identity in an open 
neighborhood of the boundary of $D^2$ and changes sign elsewhere.
\end{proof}

\begin{theorem}
Let M be a compact, connected and oriented 2-manifold with smooth boundary. Then any smooth function
 f is the Gaussian curvature of some Riemannian metric on M.
\end{theorem} 

\begin{proof}
 By Lemma 2.2. there exists an extension $\widetilde{f}$ of $f$ satisfying the
sign conditions, then by the theorem of Kazdan and warner \cite{kazdan}, there exists a metric on 
$\widetilde{M}$ possesses $\widetilde{f}$ as its Gaussian curvature, restriction of the metric to 
$M$ is the expected metric.  
\end{proof}
\begin{rem}
 On the boundary the function $f$ can not identified the Gaussian curvature because the dimension of
the boundary is one. However, the obtained metric $g$ induces a metric $\jmath^*g$ (ordinary distance function)
to the boundary. The induced metric determines the geodesic curvature $k_{(\jmath^*g)}$ of the boundary explicitly. The geodesic curvature
of boundary is any smooth function $k$ on the boundary which satisfies 
$$\int _{\partial M}k = 2\pi\chi(M) -\int_{M}f dA.$$    
\end{rem}

 \end{document}